\documentclass[reqno]{amsart}
\usepackage{amsmath}
\usepackage{amscd,amsthm,amsfonts,amsopn,amssymb,mathrsfs}
\usepackage{epsfig,hyperref}

\newtheorem{theorem}{Theorem}[section]

\newtheorem{proposition}[theorem]{Proposition}
\newtheorem{corollary}[theorem]{Corollary}

\newtheorem{lemma}[theorem]{Lemma}
\theoremstyle{remark}

\DeclareMathOperator{\supp}{supp\,}

\def\XXint#1#2#3{{\setbox0=\hbox{$#1{#2#3}{\int}$ }
\vcenter{\hbox{$#2#3$ }}\kern-.6\wd0}}

\allowdisplaybreaks
\begin{document}

\title[NLW with Partial Symmetry]{Global well-posedness for the logarithmically energy-supercritical Nonlinear Wave Equation with partial symmetry}
\author[A. Bulut]{Aynur Bulut}
\address{Department of Mathematics, Louisiana State University}
\email{aynurbulut@lsu.edu}
\author[B.Dodson]{Benjamin Dodson}
\address{Department of Mathematics, Johns Hopkins University}
\email{bdodson4@jhu.edu}
\date{July 9, 2018}

\begin{abstract}
We establish global well-posedness and scattering results for the logarithmically energy-supercritical nonlinear wave
equation, under the assumption that the initial data satisfies a partial symmetry condition.  These results generalize
and extend work of Tao in the radially symmetric setting.  The techniques involved include weighted versions of Morawetz 
and Strichartz estimates, with weights adapted to the partial symmetry assumptions.  

In an appendix, we establish a corresponding quantitative result for the energy-critical problem.
\end{abstract}

\maketitle

\section{Introduction}

The goal of this paper is to show how partial symmetry assumptions on initial data can lead to enhanced global well-posedness
results for nonlinear wave equations posed on Euclidean spaces $\mathbb{R}^d$, $d\geq 4$.  For the sake of simplicity, we
restrict our considerations to $\mathbb{R}^4$, however we expect that our results can be extended to higher dimensional settings
without much difficulty.  We focus our attention on the nonlinear wave equation,
\begin{align*}
\textrm{(NLW)}\quad \left\lbrace\begin{array}{c}u_{tt}-\Delta u+F(u)=0,\quad (t,x)\in \mathbb{R}\times\mathbb{R}^4\\
(u,u_t)|_{t=0}=(u_0,u_1),
\end{array}\right.
\end{align*}
with defocusing nonlinearity $F$ in the {\it slightly energy-supercritical} regime,
\begin{align}
F:u\mapsto u^3\log(3+u^2).\label{3.1}
\end{align}
To explain this terminology, we recall that when $F$ is a defocusing power-type nonlinearity given by $F(u)=|u|^pu$, the problem is {\it energy-subcritical} for $p<2$ and {\it energy-critical} for $p=2$; in these two cases, solutions (starting, e.g. from sufficiently regular and decaying initial data) are known to be globally well-posed, with a corresponding scattering result, while when $p>2$ the problem is {\it energy-supercritical}, and the long-time behavior of solutions remains a prominent open question (see \cite{KM1,KM2,KV1,KV2,B1,B2,B3,DL,DR}, and references cited in these works, for results showing that {\it a priori} control over a critical norm implies global well-posedness, as well as \cite{TaoSC} where energy-supercritical blowup is shown to be possible for certain {\it systems} of defocusing nonlinear wave equations; see also \cite{BeSo,BeSo2,DKM,DS,KrSch} for other results in energy-supercritical settings).
\setlength{\parskip}{0.3\baselineskip}

Here, a key ingredient in long-time control over solutions is the {\it energy} associated to (NLW),
\begin{align*}
E[u,u_t]=\frac{1}{2}\int_{\mathbb{R}^4} |\nabla u|^2+|u_t|^2dx+\int_{\mathbb{R}^4} G(u)dx,
\end{align*}
where $G$ is given by $$G(u)=\int_0^u F(t)dt,$$ and which is conserved in time for solutions of (NLW).  When $F$ is given by ($\ref{3.1}$), one has $G(u)\approx u^4\log(3+u^2)$, while when $F$ is of power-type $|u|^pu$, $G(u)=\frac{1}{p+2}|u|^{p+2}$.

In recent years, beginning with work of Tao \cite{Tao}, several authors have studied the global well-posedness and scattering problem for various cases of the three-dimensional nonlinear wave equation with slightly energy-supercritical nonlinearity, obtaining striking results which show that the global well-posedness theory can be extended from the energy-subcritical and energy-critical settings into the slightly energy-supercritical regime (that is, admitting the inclusion of logarithmic factors).  While in the discussion that follows we focus on (NLW), we make note of the works \cite{R1,R2,R3} on the nonlinear Schr\"odinger equation, and \cite{TaoNS,KT,BMR} on variants of the Navier-Stokes system.

In \cite{Tao}, Tao established global well-posedness for (NLW) posed on $\mathbb{R}^3$ with $F(u)=u^5\log(3+u^2)$, under the assumption that the initial data $(u_0,u_1)$ has radial symmetry (note that in this three-dimensional setting, the energy-critical power-type nonlinearity is $F(u)=u^5$).  The technique in \cite{Tao} is based on a quantitative scattering bound obtained by Ginibre, Soffer, and Velo in \cite{GSV}, valid for radially symmetric solutions, which relies on a combination of the Morawetz estimate with decay properties guaranteed by the radial case of the Sobolev embedding.  See also work of Shih \cite{Shih}, where the method is refined to treat $F(u)=u^5\log^{c}(3+u^2)$, $0<c\leq 4/3$.

The nonradial case was studied by Roy \cite{RoyNLW}.  The result in \cite{RoyNLW} (see also \cite{Roy-NLW-supercritical}) is based on a {\it nonradial} quantitative form \cite{Tao-quant} of the energy-critical global well-posedness result for three-dimensional NLW (closely related to the induction on energy technique of Bourgain \cite{Bourgain}, originally developed in the nonlinear Schr\"odinger setting).  Whereas the radial energy-critical bound of \cite{GSV} gives control which is polynomial in the energy, the nonradial result of \cite{Tao-quant} gives an estimate which exhibits double-exponential dependence on the energy.  As a consequence, the results in \cite{RoyNLW} apply to nonlinearities of ``log-log supercritical'' type, $F(u)=u^5\log^c(\log(3+u^2))$ for $c$ sufficiently small (in particular, \cite{RoyNLW} treats the case $c\in (0,8/225)$), while ``log-supercritical'' nonlinearities of type (\ref{3.1}) have so far remained out of reach in the non-radial case.

In the present paper, we consider (NLW) on $\mathbb{R}^4$, and show that {\it partial} symmetry assumptions can lead to results which apply to logarithmically supercritical nonlinearities, giving an improvement over general non-radial methods as used in \cite{RoyNLW}.  For comparison, we begin with the $\mathbb{R}^4$ analog of Tao's three-dimensional log-energy-supercritical result.  To fix notation, for $s>1$ let $$\tilde{H}^s(\mathbb{R}^4):=\dot{H}^s(\mathbb{R}^4)\cap \dot{H}^1(\mathbb{R}^4).$$

\begin{theorem}[Radial log-supercritical NLW on $\mathbb{R}^4$]
\label{p.1}
Suppose that $u:I\times\mathbb{R}^4\rightarrow\mathbb{R}$ is a solution to (NLW) with nonlinearity (\ref{3.1}) and radially symmetric initial data $(u_0,u_1)\in \tilde{H}^{9/4}(\mathbb{R}^4)\times H^{5/4}(\mathbb{R}^4)$.  Then 
\begin{align*}
\lVert (u,u_t)\rVert_{L_t^\infty(\tilde{H}_x^{9/4}\times H_x^{5/4})}\lesssim 1,
\end{align*}
with constant independent of the time interval $I$.  In particular, combining this estimate with the usual local theory for (NLW), radial solutions to (NLW) with nonlinearity (\ref{3.1}) exist globally in time and scatter at $\pm\infty$.
\end{theorem}

For the convenience of the reader, we sketch a proof of Theorem $\ref{p.1}$ in Section $\ref{s.prop1}$ below.

Our main result is the following theorem, which shows that this global well-posedness property persists when the assumption of radial initial data is relaxed to the assumption that $u_0$ and $u_1$ obey a symmetry condition of the form
\begin{align}
u(x)=\tilde{u}(|x'|,x_4),\quad x\in\mathbb{R}^4,\label{symm1}
\end{align}
with $\tilde{u}:\mathbb{R}_+\times\mathbb{R}\rightarrow\mathbb{R}$, where $x'=(x_1,x_2,x_3)$ for $x=(x_1,\ldots,x_4)\in\mathbb{R}^4$.

\begin{theorem}[Axially symmetric log-supercritical NLW on $\mathbb{R}^4$]
\label{t.1}
Suppose that $u:I\times\mathbb{R}^4\rightarrow\mathbb{R}$ is a solution to (NLW) with nonlinearity (\ref{3.1}) and initial data $(u_0,u_1)\in \tilde{H}^{5/2}(\mathbb{R}^4)\times H^{3/2}(\mathbb{R}^4)$ with each of $u_0$ and $u_1$ satisfying the symmetry condition (\ref{symm1}).  Then 
\begin{align*}
\lVert (u,u_t)\rVert_{L_t^\infty(\tilde{H}_x^{5/2}\times H_x^{3/2})}\lesssim 1,
\end{align*}
with constant independent of the time interval $I$.  In particular, solutions to (NLW) with nonlinearity (\ref{3.1}) and initial data having axial symmetry of the form (\ref{symm1}) exist globally in time and scatter at $\pm\infty$.
\end{theorem}

The proof of Theorem $\ref{t.1}$ is based on a bootstrap procedure and continuity argument.  The key long-time estimates are provided by a variant of the Morawetz estimate adapted to weights in the symmetric variables, complemented with a class of weighted Strichartz estimates.  These weighted estimates of Morawetz and Strichartz type are the new ingredients which allow us to fully exploit the anisotropic decay satisfied by solutions in our setting (that is, decay arising from the symmetry assumption ($\ref{symm1}$)).

We remark that our techniques can be extended to other partially symmetric settings.  As an example, we give a related global well-posedness result when the initial data $u_0$ and $u_1$ has product-type symmetry
\begin{align}
u(x)=\tilde{u}(|(x_1,x_2)|,|(x_3,x_4)|),\quad x=(x_1,\ldots,x_4)\in\mathbb{R}^4,\label{symm2}
\end{align}
with $\tilde{u}:\mathbb{R}^2_+\rightarrow\mathbb{R}$.

\begin{theorem}
\label{t.2}
Suppose that $u:I\times\mathbb{R}^4\rightarrow\mathbb{R}$ is a solution to (NLW) with nonlinearity (\ref{3.1}) and initial data $(u_0,u_1)\in \tilde{H}^{5/2}(\mathbb{R}^4)\times H^{3/2}(\mathbb{R}^4)$ with each of $u_0$ and $u_1$ satisfying the symmetry condition (\ref{symm2}).  Then 
\begin{align*}
\lVert (u,u_t)\rVert_{L_t^\infty(\tilde{H}_x^{5/2}\times H_x^{3/2})}\lesssim 1,
\end{align*}
with constant independent of the time interval $I$.  In particular, solutions to (NLW) with nonlinearity (\ref{3.1}) and initial data having symmetry of the form (\ref{symm2}) exist globally in time and scatter at $\pm\infty$.
\end{theorem}

As we alluded to in our discussion of the works of Tao \cite{Tao} and Roy \cite{RoyNLW} above, there is a close connection between global well-posedness results for slightly energy-supercritical (NLW) and quantitative estimates for the energy-critical problem.  This viewpoint carries over to our partially symmetric setting, and we include one such result in Appendix A.  In particular, for solutions to energy-critical (NLW) satisfying the symmetry condition (\ref{symm1}), we present a global well-posedness and scattering result with bounds that have {\it polynomial} dependence on the energy.

To the best of our knowledge, Theorem \ref{t.1} and Theorem \ref{t.2} (as well as the energy-critical results discussed in the appendix) are among the first instances where partial symmetry assumptions are exploited to improve global well-posedness results for (NLW); in this context, we note also the works of Martel \cite{Martel} which uses a similar philosophy to establish blow-up results for the nonlinear Schr\"odinger equation, as well as Liu--Wang \cite{LW} on axisymmetric Navier-Stokes.  We expect that this philosophy can be broadly applied in settings where radial symmetry arises as a useful hypothesis.  We plan to revisit the question of extending these results to a three-dimensional setting in a future work.

\subsection*{Outline}

We now describe the structure of the rest of this article.  In Section $2$, we prove the radial $\mathbb{R}^4$ result, Theorem $\ref{p.1}$, with an argument (as in \cite{Tao}) based on Strichartz estimates, the Morawetz estimate, and the radial Sobolev embedding.  The subsequent Sections $3$--$5$ deal with our main results concerning global well-posedness under partial symmetry assumptions.  In Section $3$ we establish a class of weighted Strichartz estimates for solutions to (NLW) under the symmetry assumptions ($\ref{symm1}$) and ($\ref{symm2}$).  These are used in Section $4$ to prove Theorem $\ref{t.1}$, where they are combined with a suitable form of the Morawetz estimate adapted to the symmetry condition ($\ref{symm1}$), and, similarly in Section $5$ to prove Theorem $\ref{t.2}$.  In the appendix, we establish an associated result for the energy-critical problem.

\subsection*{Acknowledgements}

The work of A.B. was partially supported by NSF Grants DMS-1361838 and DMS-1748083.  The work of B.D. was partially supported by NSF Grant DMS-1500424.  Both authors are grateful to A. Soffer for enlightening discussions and comments on an earlier draft.

\section{Proof of Theorem \ref{p.1}}
\label{s.prop1}

In this section, we prove the radial global well-posedness result, Theorem $\ref{p.1}$.  The argument is largely similar to the
three-dimensional case of \cite{Tao} and \cite{Shih}, and is based on Strichartz estimates, the Morawetz estimate for (NLW), and 
the radial Sobolev embedding.  

We begin by specifying a few notational conventions: we write $A\lesssim B$ to mean $A\leq CB$ for 
some constant $C>0$, and allow for the constants $C$ to change from line to line unless otherwise indicated.  For $s\geq 0$, we 
use $H_x^s$ and $\dot{H}_x^s$ to refer to the usual ($L^2$-based) inhomogeneous and homogeneous Sobolev spaces, respectively.  In
addition, we use subscripts on the spaces $L^p$, $H^s$ and $\dot{H}^s$ to indicate the variables of integration for the appropriate
norm.

\subsection{Strichartz estimates for (NLW)}

We recall the usual Strichartz estimates for solutions to the inhomogeneous wave equation (NLW) on $\mathbb{R}^4$.  These estimates read
\begin{align*}
\lVert |\nabla|^s u\rVert_{L_t^qL_x^r}\lesssim \lVert (u(0),u_t(0))\rVert_{\dot{H}_x^\mu\times \dot{H}_x^{\mu-1}}+\lVert |\nabla|^{\tilde{s}}F\rVert_{L_t^{a'}L_x^{b'}},
\end{align*}
where $(q,r),(a,b)\in [2,\infty]\times [2,\infty)$ satisfy
\begin{align*}
\frac{1}{q}+\frac{4}{r}=2-(\mu-s),\quad \frac{1}{q}+\frac{3}{2r}\leq \frac{3}{4},
\end{align*}
\begin{align*}
\frac{1}{a}+\frac{4}{b}=2-(1+\tilde{s}-\mu),\quad \frac{1}{a}+\frac{3}{2b}\leq \frac{3}{4},
\end{align*}
\begin{align*}
\frac{1}{a}+\frac{1}{a'}=1,\quad \frac{1}{b}+\frac{1}{b'}=1.
\end{align*}

For the convenience of the reader, we record the particular instances of these estimates that we use below.  In section 2, we use
\begin{align*}
\lVert (u,u_t)\rVert_{L_t^\infty(I;\dot{H}_x^{9/4}\times \dot{H}_x^{5/4})}&\lesssim \lVert (u(t_0),u_t(t_0))\rVert_{\dot{H}_x^{9/4}\times \dot{H}_x^{5/4}}+\lVert D^2F\rVert_{L_t^{20/11}L_x^{5/4}},\\
\lVert \nabla u\rVert_{L_t^{20/7}L_x^{10}}&\lesssim \lVert (u(t_0),u_t(t_0)\rVert_{\dot{H}_x^{9/4}\times \dot{H}_x^{5/4}}+\lVert D^2F\rVert_{L_{t}^{20/11}L_x^{5/4}},\\
\lVert D^2u\rVert_{L_t^{20/3}L_x^{5/2}}&\lesssim \lVert (u(t_0),u_t(t_0)\rVert_{\dot{H}_x^{9/4}\times \dot{H}_x^{5/4}}+\lVert D^2F\rVert_{L_t^{20/11}L_x^{5/4}}.
\end{align*}

\subsection{Morawetz estimate and the radial Sobolev embedding}

The usual Morawetz estimate for (NLW) (on $\mathbb{R}^4$) with nonlinearity given by ($\ref{3.1}$) is
\begin{align*}
\int_I\int_{\mathbb{R}^4} \frac{u(t,x)^4\log(3+u(t,x)^2)}{|x|}dxdt\leq CE[u_0,u_1]
\end{align*}
for some constant $C>0$ independent of $u$, where $u:I\times\mathbb{R}^4\rightarrow \mathbb{R}$ is a solution of (NLW).  Combining this with the radial Sobolev bound 
\begin{align}
\lVert |x|u(t,x)\rVert_{L_{t,x}^\infty}\lesssim \lVert u\rVert_{\dot{H}_x^1(\mathbb{R}^4)}\label{radSobolev}
\end{align}
(see, for instance, \cite{Ni}), one obtains
\begin{align*}
\int_I\int_{\mathbb{R}^4} |u(t,x)|^5\log(3+u(t,x)^2)dxdt\leq CE[u_0,u_1]^{3/2}.
\end{align*}

\subsection{Sketch of the Proof of Theorem \ref{p.1}}

As in \cite{Tao}, the argument used to prove Theorem $\ref{p.1}$ is based on an iterative application of Strichartz estimates.  We sketch the key estimate in this section, a conditional bound on the Strichartz norm for short-time intervals, which is then used inductively to establish the result.  We postpone the details of this inductive argument to Section $5$, where it is performed in the setting of the proof of Theorem $\ref{t.1}$.

Fix an interval $I=[t_0,t_1]\subset I_{\textrm{max}}$, where $I_{\textrm{max}}$ denotes the maximal interval of existence, and define
\begin{align*}
A(t_0,t_1)&:=\int_{t_0}^{t_1}\int_{\mathbb{R}^4} |u(t,x)|^5\log(3+u(t,x)^2)dxdt,
\end{align*}
along with
\begin{align*}
Z(t)&:=\lVert \nabla u\rVert_{L_t^{20/7}([t_0,t];L_x^{10})}+\lVert D^2u\rVert_{L_t^{20/3}([t_0,t];L_x^{5/2})}\\
&\hspace{0.8in}+\sup_{t_0\leq t'\leq t} \lVert (u(t'),u_t(t'))\rVert_{\dot{H}^{9/4}_x(\mathbb{R}^4)\times \dot{H}_x^{5/4}(\mathbb{R}^4)}, \quad t\in I.
\end{align*}

Let $t\in I$ be given.  By Strichartz,
\begin{align}
\nonumber Z(t)&\lesssim \lVert u(t_0)\rVert_{\dot{H}^{9/4}_x(\mathbb{R}^4)}+\lVert u_t(t_0)\rVert_{\dot{H}^{5/4}_x(\mathbb{R}^4)}\\
\nonumber &\hspace{0.8in}+\lVert D_{x}^2[u^3\log(3+u^2)]\rVert_{L_t^{20/11}L_x^{5/4}}\\
\nonumber &\lesssim Z(t_0)+\lVert u|\nabla_x u|^2\log(3+u^2)\rVert_{L_t^{20/11}L_x^{5/4}}\\
&\hspace{0.8in}+\lVert u^2(D^2u)\log(3+u^2)\rVert_{L_t^{20/11}L_x^{5/4}},\label{7.1}
\end{align}
where we have used the bounds $u^3|\nabla u|^2/(3+u^2)\lesssim u|\nabla u|^2\lesssim u|\nabla u|^2\log(3+u^2)$ and $u^5|\nabla u|^2/(3+u^2)^2\lesssim u|\nabla u|^2\lesssim u|\nabla u|^2\log(3+u^2)$.

Using the H\"older inequality, the right-hand side of (\ref{7.1}) is bounded by a multiple of
\begin{align*}
&Z(t_0)+\lVert \nabla u\rVert_{L_t^\infty L_x^2}\lVert \nabla u\rVert_{L_t^{20/7}L_x^{10}}\lVert u\log^{1/5}(3+u^2)\rVert_{L_{t,x}^5}\log^{4/5}(3+\lVert u\rVert_{L_{t,x}^{\infty}}^2)\\
&\hspace{0.6in}+\lVert D^2u\rVert_{L_t^{20/3}L_x^{5/2}}\lVert u^2\log^{2/5}(3+u^2)\rVert_{L_{t,x}^{5/2}}\log^{3/5}(3+\lVert u\rVert_{L_{t,x}^{\infty}}^2)\\
&\hspace{0.2in} \lesssim Z(t_0)+E[u_0,u_1]^{1/2}Z(t)A(t_0,t_1)^{1/5}\log^{4/5}(3+E[u_0,u_1]+Z(t)^2)\\
&\hspace{0.8in}+Z(t)A(t_0,t_1)^{2/5}\log^{3/5}(3+E[u_0,u_1]+Z(t)^2)\\
&\hspace{0.2in} \lesssim Z(t_0)+(1+E[u_0,u_1]^{1/2})Z(t)A(t_0,t_1)^{1/5}\log^{4/5}(3+E[u_0,u_1]+Z(t)^2),
\end{align*}
provided $A(t_0,t_1)\leq 1$, where we have also used the bound $\log(3+E[u_0,u_1]+Z(t)^2)\geq 1$.

A standard continuity argument now shows that if one has a bound of the form $A(t_0,t_1)\leq \frac{\epsilon}{\log^{4}(3+E[u_0,u_1]+Z(t_0)^2)}$ with $\epsilon$ sufficiently small (depending on $(u_0,u_1)$), then one can conclude
\begin{align*}
Z(t)\leq CZ(t_0),\quad t\in I,
\end{align*}
as desired.  To finish the proof of Theorem \ref{p.1}, one argues as in \cite{Tao}, appealing to a partitioning argument based on dividing a given time interval $[0,T]$ into a collection of subintervals $[t_i,t_{i+1}]$, $i=0,...,m$, on which the desired control on $A(t_i,t_{i+1})$ holds.  Since we give a full discussion of a closely related variant of this argument in Section \ref{s.5} below, we omit the details.

\section{Weighted Strichartz estimates with symmetry}

In this section, as preparation for our proofs of Theorem $\ref{t.1}$ and Theorem $\ref{t.2}$, we prove several weighted Strichartz estimates adapted to the symmetry conditions ($\ref{symm1}$) and ($\ref{symm2}$).  In fact, both of the relevant estimates originate in a Strichartz bound for a third symmetry condition,
\begin{align}
u(x) = \tilde{u}(|(x_{1}, x_{2})|, x_{3}, x_{4}).\label{symm3}
\end{align}
with $\tilde{u}:\mathbb{R}_+\times\mathbb{R}\times\mathbb{R}\rightarrow\mathbb{R}$, expressed in the following lemma.

\begin{lemma}\label{l2.1}
Suppose that $u$ is a solution to (NLW) which satisfies the symmetry assumption (\ref{symm3}) for all $t \in I$.
Then one has the estimate
\begin{align*}
\| (x_{1}^{2} + x_{2}^{2})^{1/8} u \|_{L_{t,x}^{4}(I \times \mathbb{R}^{4})}&\lesssim \| u_{0} \|_{\dot{H}^{1/2}} + \| u_{1} \|_{\dot{H}^{-1/2}} \\
&\hspace{1.2in}+ \| (x_{1}^{2} + x_{2}^{2})^{-1/8} F \|_{L_{t,x}^{4/3}(I \times \mathbb{R}^{4})}.
\end{align*}
\end{lemma}

\begin{proof}
To simplify notation, for $x=(x_1,x_2,x_3,x_4)\in\mathbb{R}^4$, let $r=(x_1^2+x_2^2)^{1/2}$, $w=(x_1,x_2)\in\mathbb{R}^2$ and $y=(x_3,x_4)\in\mathbb{R}^2$.  Let $\xi$ denote the Fourier variable dual to $x$, and for each $\xi\in\mathbb{R}^4$ set $s=(\xi_1^2+\xi_2^2)^{1/2}$, $u=(\xi_1,\xi_2)$, and $v=(\xi_3,\xi_4)$.  Moreover, we write $u(r,y)=u(w,y)=u(x)$ for $x\in\mathbb{R}^4$ and $\widehat{u}(s,v)=\widehat{u}(u,v)=\widehat{u}(\xi)$ for $\xi\in\mathbb{R}^4)$.

We perform decompositions in both space and frequency.  Let $\phi\in C_c^\infty(\mathbb{R}^2)$ be radially symmetric and such that $\phi(x)=1$ for $|x|\leq 1$ and $\supp \phi\subset\{x:|x|\leq 2\}$.  For $k\in\mathbb{Z}$, let $\chi_k:\mathbb{R}^2\rightarrow\mathbb{R}$ be the characteristic function of the ball $\{w\in\mathbb{R}^2:2^k\leq |w|\leq 2^{k+1}\}$, and set $\phi_k(x)=\phi(2^{-k}x)$, $\psi(x)=\phi(x)-\phi(2x)$, and $\psi_k(x)=\psi(2^{-k}x)$ for $x\in\mathbb{R}^2$.  Moreover, for $f\in\mathcal{S}(\mathbb{R}^4)$, let $P_kf$ and $P_{>k}f$ be defined by $\widehat{P_kf}(\xi)=\psi_k(u)\widehat{f}(\xi)$ and $\widehat{P_{>k}f}(\xi)=\sum_{j>k}\phi_j(u)\widehat{f}(\xi)$, respectively.  Then
\begin{align}
\nonumber &\lVert r^{1/4}e^{it|\nabla|}u_0\rVert_{L_{t,x}^4}\\
\nonumber &\hspace{0.2in}=\bigg(\sum_{k\in\mathbb{Z}}\iint r\chi_k(r)|e^{it|\nabla|}u_0(t,x)|^4dxdt\bigg)^{1/4}\\
\nonumber &\hspace{0.2in}\lesssim \bigg(\sum_{k\in\mathbb{Z}} \iint r^2\chi_k(r)|e^{it|\nabla_x|}u_0(r,y)|^4drdydt\bigg)^{1/4}\\
\nonumber &\hspace{0.2in}\lesssim \bigg(\sum_{k\in\mathbb{Z}} 2^{2k}\lVert \chi_k(r)e^{it|\nabla_x|}u_0(r,y)\rVert_{L_{t,r,y}^4}^4\bigg)^{1/4}\\
\nonumber &\hspace{0.2in}\lesssim \bigg(\sum_k\Big[\sum_{j\leq -k} 2^{k/2}\lVert \chi_k(r)e^{it|\nabla_x|}P_ju_0(r,y)\rVert_{L_{t,r,y}^4}\\
&\hspace{1.2in}+2^{k/2}\lVert \chi_k(r)e^{it|\nabla_x|}P_{>-k}u_0(r,y)\rVert_{L_{t,r,y}^4}\Big]^4\bigg)^{1/4}.\label{3.4}
\end{align}

We begin by estimating the contributions of terms corresponding to $(k,j)$ with $j\leq -k$.  For such terms, one has
\begin{align}
\nonumber &\lVert \chi_k(r)e^{it|\nabla_x|}P_ju_0(r,y)\rVert_{L_{t,r,y}^4}\\
\nonumber &\hspace{0.2in}=\bigg\lVert \int_0^{2\pi}\int \chi_k(r)se^{irs\cos(\theta)+iv\cdot y+it(s^2+|v|^2)^{1/2}}\psi_j(s)\widehat{u_0}(s,v)dsdvd\theta\bigg\rVert_{L_{t,r,y}^4}\\
\nonumber &\hspace{0.2in}\lesssim \int_0^{2\pi}\bigg\lVert \int se^{irs\cos(\theta)+iv\cdot y+it(s^2+|v|^2)^{1/2}}\psi_j(s)\widehat{u_0}(s,v)dsdv\bigg\rVert_{L_{t,r,y}^4}d\theta\\
&\hspace{0.2in}\lesssim \int_0^{2\pi}|\cos(\theta)|^{-1/4}\bigg\lVert \int e^{irs+iv\cdot y+it(s^2+|v|^2)^{1/2}}s\psi_j(s)\widehat{u_0}(s,v)dsdv\bigg\rVert_{L_{t,r,y}^4}d\theta.\label{3.3}
\end{align}
where we've used the change of variables $r\mapsto r/\cos(\theta)$ in the last inequality.  Invoking three-dimensional Strichartz estimates for the wave equation, we obtain
\begin{align*}
(\ref{3.3})&\lesssim \int_0^{2\pi}|\cos(\theta)|^{-1/4}\lVert (s^2+|v|^2)^{1/4}s\psi_j(s)\widehat{u_0}(s,v)\rVert_{L_{s,v}^2}d\theta\\
&\lesssim \int_0^{2\pi} |\cos(\theta)|^{-1/4}\lVert (|u|^2+|v|^2)^{1/4}|u|^{1/2}\psi_j(u)\widehat{u_0}(u,v)\rVert_{L_{u,v}^2}d\theta\\
&\lesssim 2^{j/2}\lVert P_ju_0\rVert_{\dot{H}_x^{1/2}}.
\end{align*}

Fixing $k\in\mathbb{Z}$, we now turn to the $P_{>-k}u_0$ contribution.  By standard estimates for oscillatory integrals, there exists $g\in C^\infty([0,2\pi])$ such that
\begin{align*}
\int_0^{2\pi} e^{isr\cos(\theta)}d\theta&=(sr)^{-1}\int_0^{2\pi} g(\theta)e^{isr\cos(\theta)}d\theta\\
&\hspace{1.2in}+C_1e^{isr}(sr)^{-1/2}+C_2e^{-isr}(sr)^{-1/2}.
\end{align*}
This gives
\begin{align*}
&\lVert \chi_k(r)e^{it|\nabla_x|}P_{>-k}u_0(r,y)\rVert_{L_{t,r,y}^4}\lesssim (I)_k+(II)_k,
\end{align*}
with
\begin{align*}
(I)_k&:=\int_0^{2\pi}\bigg\lVert \chi_k(r)r^{-1}\int e^{irs\cos(\theta)+iv\cdot y+it(s^2+|v|^2)^{1/2}}\\
&\hspace{2.2in}\cdot\psi_{>-k}(s)\widehat{u_0}(s,v)dsdv\bigg\rVert_{L_{t,r,y}^4}|g(\theta)|d\theta,
\intertext{and}
(II)_k&:=\sum_{\ell\in\{-1,1\}} 2^{-k/2}\bigg\lVert \chi_k(r)\int s^{1/2}e^{i\ell rs+iv\cdot y+it(s^2+|v|^2)^{1/2}}\\
&\hspace{2.2in}\cdot\psi_{>-k}(s)\widehat{u_0}(s,v)dsdv\bigg\rVert_{L_{t,r,y}^4}
\end{align*}
where we have used the definition $\psi_{>-k}:=\sum_{j>-k} \psi_k$.

To estimate $(I)_k$, note that, as before, using the change of variables $r\mapsto r/\cos(\theta)$ and invoking three-dimensional Strichartz estimates, 
\begin{align*}
(I)_k
&\lesssim 2^{-k}\sum_{j>-k}\lVert (s^2+|v|^2)^{1/4}\psi_j(s)\widehat{u_0}(s,v)\rVert_{L_{s,v}^2}\\
&\lesssim 2^{-k}\sum_{j>-k}\lVert (|u|^2+|v|^2)^{1/4}|u|^{-1/2}\psi_j(s)\widehat{u_0}(s,v)\rVert_{L_{u,v}^2}\\
&\lesssim \sum_{j>-k} 2^{-k-j/2}\lVert P_ju_0\rVert_{\dot{H}_x^{1/2}}.
\end{align*}

Collecting these bounds, we get
\begin{align}
(\ref{3.4})
&\lesssim \bigg(\sum_k\Big[\sum_{j\in\mathbb{Z}} 2^{-|j+k|/2}\lVert P_ju_0\rVert_{\dot{H}_x^{1/2}} \Big]^4\bigg)^{1/4}+\bigg(\sum_k\Big[2^{k/2}(II)_k\Big]^4\bigg)^{1/4}.\label{3.5}
\end{align}

Recalling the definition of the functions $\chi_k$, $k\in\mathbb{Z}$, as characteristic functions,
\begin{align}
\nonumber &\bigg(\sum_{k} \Big[2^{k/2}(II)_k\Big]^4\bigg)^{1/4}\\
\nonumber &\hspace{0.2in}\lesssim \sum_{\ell\in \{-1,1\}}\bigg\lVert \int s^{1/2}e^{i\ell rs+iv\cdot y+it(s^2+|v|^2)^{1/2}}\psi_{>-k}(s)\widehat{u_0}(s,v)dsdv\bigg\rVert_{L_{t,r,y}^4}\\
\nonumber &\hspace{0.2in}\lesssim \lVert (s^2+|v|^2)^{1/4}s^{1/2}\psi_{>-k}(s)\widehat{u_0}(s,v)\rVert_{L_{s,v}^2}\\
\nonumber &\hspace{0.2in}\lesssim \lVert (|u|^2+|v|^2)^{1/4}\psi_{>-k}(u)\widehat{u_0}(s,v)\rVert_{L_{u,v}^2}\\
\nonumber &\hspace{0.2in}\lesssim \lVert P_{>-k}u_0\rVert_{\dot{H}_x^{1/2}}\\
&\hspace{0.2in}\lesssim \lVert u_0\rVert_{\dot{H}_x^{1/2}}.\label{3.6}
\end{align}

Combining ($\ref{3.5}$) and ($\ref{3.6}$) and using Young's inequality, we obtain
\begin{align*}
\lVert r^{1/4}e^{it|\nabla|}u_0\rVert_{L_{t,x}^4}&\lesssim \lVert u_0\rVert_{\dot{H}_x^{1/2}}.
\end{align*}
The desired result now follows from duality considerations and the Christ-Kiselev lemma.
\end{proof}

We now give a weighted Strichartz estimate associated to the symmetry condition ($\ref{symm1}$).

\begin{lemma}\label{l1.3}
If $u$ is a solution the wave equation (NLW) on a time interval $I$ which satisfies
the symmetry condition (\ref{symm1}), i.e.
\begin{align*}
u(t, x_{1}, x_{2}, x_{3}, x_{4}) = u(t, x_{1}^{2} + x_{2}^{2} + x_{3}^{2}, x_{4}),
\end{align*}
then one has the estimate
\begin{align}
\| |x'|^{1/4} u \|_{L_{t,x}^{4}(I \times \mathbb{R}^{4})} \lesssim \| u_{0} \|_{\dot{H}^{1/2}(\mathbb{R}^{4})} + \| u_{1} \|_{\dot{H}^{-1/2}(\mathbb{R}^{4})} + \| |x'|^{-1/4} F \|_{L_{t,x}^{4/3}(I \times \mathbb{R}^{4})},\label{1.30}
\end{align}
where $x=(x_1,x_2,x_3,x_4)\in\mathbb{R}^4$ and $x'=(x_1,x_2,x_3)$.
\end{lemma}

\begin{proof}
Let $S(t)$ be the solution operator to the wave equation, that is, $S(t)(f, g) = u(t)$ is the solution to the linear wave equation
\begin{align*}
u_{tt} - \Delta u = 0, \hspace{5mm} u(0) &= f, \hspace{5mm} u_{t}(0) = g.
\end{align*}

Now, by Lemma $\ref{l2.1}$,
\begin{align}
\nonumber &\| (x_{1}^{2} + x_{2}^{2})^{1/8} u \|_{L_{t,x',y}^{4}} + \| (x_{1}^{2} + x_{3}^{2})^{1/8} u \|_{L_{t,x',y}^{4}} \\
&\hspace{0.4in}+ \| (x_{2}^{2} + x_{3}^{2})^{1/8} u \|_{L_{t,x',y}^{4}} \lesssim \| f \|_{\dot{H}^{1/2}} + \| g \|_{\dot{H}^{-1/2}}.\label{1.32}
\end{align}
In view of this, we observe that the dual of the estimate $(\ref{1.32})$ is
\begin{align*}
\left\| \int S(-t)(0, F) dt \right\|_{\dot{H}^{1/2} \times \dot{H}^{-1/2}} \lesssim \| |x'|^{-1/4} F \|_{L_{t,x',y}^{4/3}}.
\end{align*}
The desired estimate $(\ref{1.30})$ now follows by the Christ-Kiselev lemma.
\end{proof}

As another corollary which will be useful in the next section, we obtain a Strichartz bound involving second derivatives.

\begin{corollary}\label{c1}
For solutions $u$ of (NLW) as in Lemma \ref{l1.3}, one has the estimate
\begin{align}
\nonumber \| |x'|^{1/4} D^{2} u \|_{L_{t,x}^{4}(I \times \mathbb{R}^{4})}&\lesssim \| u_{0} \|_{\dot{H}^{5/2}(\mathbb{R}^{4})} + \| u_{1} \|_{\dot{H}^{3/2}(\mathbb{R}^{4})} \\
&\hspace{0.2in}+ \| |x'|^{-1/4} D^{2} F \|_{L_{t,x}^{4/3}(I \times \mathbb{R}^{4})}.\label{1.34}
\end{align}
\end{corollary}

\section{Proof of Theorem \ref{t.1}}
\label{s.5}

In this section, we prove our main result, Theorem $\ref{t.1}$, on global well-posedness for solutions with initial data satisfying the symmetry condition (\ref{symm1}), that is, $$u(x)=\widetilde{u}(|x'|,x_4),\quad x\in\mathbb{R}^4,$$ with $x'=(x_1,x_2,x_3)\in\mathbb{R}^3$.  We first introduce a variant of the Morawetz estimate adapted to this axially symmetric setting.

\begin{proposition}
\label{p.morawetz}
If $u$ solves (NLW) on $I\times\mathbb{R}^4$ and satisfies the symmetry condition (\ref{symm1}), then
\begin{align*}
\iint_{I\times\mathbb{R}^4} \frac{u(t,x)^4\log(3+u(t,x)^2)}{|x'|}dtdx\lesssim E[u_0,u_1].
\end{align*}
\end{proposition}

\begin{proof}
Define the Morawetz potentials,
\begin{align*}
M_{1}(t) = \int u_{t}(t,x', y) \frac{x'}{|x'|} \cdot \nabla_{x'} u(t,x', y) dx' dy
\end{align*}
and
\begin{align*}
M_{2}(t) = \int u_{t}(t,x',y) u(t,x',y) \frac{1}{|x'|} dx' dy.
\end{align*}

This leads to 
\begin{align*}
\frac{d}{dt} M_{1}(t)&=\int u_{tt}(t,x',y) \frac{x'}{|x'|} \cdot \nabla_{x'} u(t,x',y) dx' dy \\
&\hspace{0.2in}+ \int u_{t}(t,x',y) \frac{x'}{|x'|} \cdot \nabla_{x'} u_{t}(t,x',y) dx' dy,
\end{align*}
so that, integrating by parts and using (NLW), 
\begin{align*}
\frac{d}{dt}M_1(t)&=-\int \frac{1}{|x'|} u_{t}(t,x',y)^{2} dx' dy + \int \Delta u(t,x',y) \frac{x'}{|x'|} \cdot \nabla_{x'} u(t,x',y) dx' dy \\
&\hspace{0.2in} - \int u^{3} \log(3 + u^{2})(t,x',y) \frac{x'}{|x'|} \cdot \nabla_{x'} u(t,x',y) dx' dy.
\end{align*}

Writing $\Delta u= \Delta_{x'}u+ \Delta_{y}u$ and integrating by parts again,
\begin{align*}
&\int \Delta u(t,x',y)\frac{x'}{|x'|}\cdot \nabla_{x'}u(t,x',y)dx'dy\\
&\hspace{0.2in}=\int \Delta_{y} u(t,x',y) \frac{x'}{|x'|} \cdot \nabla_{x'} u(t,x',y) dx' dy \\
&\hspace{0.4in}+ \int \Delta_{x'} u(t,x',y) \frac{x'}{|x'|} \cdot \nabla_{x'} u(t,x',y) dx' dy \\
&\hspace{0.2in}= -\int \nabla_{y} u(t,x',y) \frac{x'}{|x'|} \cdot \nabla_{x'} \nabla_{y} u(t,x',y) dx' dy \\
&\hspace{0.4in}+ \int \Delta_{x'} u(t,x',y) \frac{x'}{|x'|} \cdot \nabla_{x'} u(t,x',y) dx' dy,
\end{align*}
which is in turn equal to 
\begin{align*}
\int \frac{|\nabla_{y} u(t,x',y)|^{2}}{|x'|} dx' dy + \int \Delta_{x'} u(t,x',y) \frac{x'}{|x'|} \cdot \nabla_{x'} u(t,x',y) dx' dy.
\end{align*}

Next, adopting the Einstein summation convention and integrating by parts once more,
\begin{align*}
&\int \Delta_{x'} u(t,x',y) \frac{x'}{|x'|} \cdot \nabla_{x'} u(t,x',y) dx' dy \\
&\hspace{0.2in}= \int \partial_{k}^{2} u(t,x',y) \frac{x'_{j}}{|x'|} \partial_{j} u(t,x',y) dx' dy \\
&\hspace{0.2in}= -\int \partial_{k} u(t,x',y) [\frac{\delta_{jk}}{|x'|} - \frac{x'_{j} x'_{k}}{|x'|^{3}}] \partial_{j} u(t,x',y) dx' dy \\
&\hspace{0.4in}- \int \partial_{k} u(t,x',y) \frac{x'_{j}}{|x'|} \partial_{j} \partial_{k} u(t,x',y) dx' dy.
\end{align*}
Evaluating the summation, this expression is equal to
\begin{align*}
\int \partial_{k} u(t,x',y) [\frac{x'_{j} x'_{k}}{|x'|^{3}}] \partial_{j} u(t,x',y) dx' dy,
\end{align*}
which (again integrating by parts) is the same as
\begin{align*}
\int \frac{(\partial_{r_{x'}} u(t,x',y))^{2}}{|x'|} dx' dy.
\end{align*}

Therefore,
\begin{align*}
\frac{d}{dt} M_{1}(t)&=-\int \frac{1}{|x'|} u_{t}(t,x',y)^{2} dx' dy \\
&\hspace{0.2in}+ \int \frac{1}{|x'|} |\nabla_{y} u(t,x',y)|^{2} dx' dy \\
&\hspace{0.2in}+ \int \frac{1}{|x'|} (\partial_{r_{x'}} u(t,x',y))^{2} dx' dy \\
&\hspace{0.2in}- \int u^{3} \log(3 + u^{2})(t,x',y) \frac{x'}{|x'|} \cdot \nabla_{x'} u(t,x',y) dx' dy.
\end{align*}

We next estimate the derivative of $M_2$, writing
\begin{align*}
\frac{d}{dt} M_{2}(t)&= \int (\Delta u - u^{3} \log(3 + u^{2}))(t,x',y) u(t,x',y) \frac{1}{|x'|} dx' dy\\
&\hspace{0.2in}+ \int u_{t}(t,x',y)^{2} \frac{1}{|x'|} dx' dy. 
\end{align*}
Again writing $\Delta u= \Delta_{x'}u+ \Delta_{y}u$ and integrating by parts,
\begin{align*}
\int \Delta u(t,x',y) \frac{1}{|x'|} u(t,x',y) dx' dy = -\int |\nabla_{y} u(t,x',y)|^{2} \frac{1}{|x'|} dx' dy \\ - \int |\nabla_{x'} u(t,x',y)|^{2} \frac{1}{|x'|} dx' dy + \frac{1}{2} \int u(t,x',y)^{2} \Delta_{x'}(\frac{1}{|x'|}) dx' dy.
\end{align*}

Now, since $\Delta_{x'}(\frac{1}{|x'|}) \leq 0$,
\begin{align*}
\frac{d}{dt} (M_{1}(t) + M_{2}(t))&\leq -u^{4} \log(3 + u^{2})(t,x',y) \frac{1}{|x'|} dx' dy \\
&\hspace{0.2in}- \int u^{3} \log(3 + u^{2})(t,x',y) \frac{x'}{|x'|} \cdot \nabla_{x'} u(t,x',y) dx' dy.
\end{align*}
Integrating by parts and using
\begin{align*}
\int_{0}^{u} (x')^{3} \log(3 + (x')^{2}) dx' \leq \frac{1}{4} u^{4} \log(3 + u^{2}),
\end{align*}
one obtains
\begin{align*}
\frac{d}{dt} (M_{1}(t) + M_{2}(t)) \leq -\frac{1}{2} \int \log(3 + u^{2}) u^{4} \frac{1}{|x'|} dx' dy.
\end{align*}

Therefore, by the fundamental theorem of calculus,
\begin{align*}
\int_{I} \int \frac{u^{4} \log(3 + u^{2})(t,x',y)}{|x'|} dx' dy dt \lesssim \sup_{t \in I} |M_{1}(t)| + |M_{2}(t)|.
\end{align*}

Now, since $\frac{x'}{|x'|} \leq 1$, one has
\begin{align*}
|M_{1}(t)| \leq \| u_{t}(t,x',y) \|_{L^{2}(\mathbb{R}^{4})} \| \nabla_{x'} u(t,x',y) \|_{L^{2}(\mathbb{R}^{4})} \leq \| u_{t} \|_{L^{2}(\mathbb{R}^{4})} \| \nabla u \|_{L^{2}(\mathbb{R}^{4})}.
\end{align*}

Also, by Hardy's inequality, for any $y \in \mathbb{R}$,
\begin{align*}
\left\| \frac{1}{|x'|} u(t,x',y) \right\|_{L_{x'}^{2}(\mathbb{R}^{3})} \lesssim \| \nabla_{x'} u(t, \cdot, y) \|_{L_{x'}^{2}(\mathbb{R}^{3})},
\end{align*}
so that
\begin{align*}
\int \frac{|u(t,x',y)|^2}{|x'|^2} dx' dy \lesssim \| \nabla u(t) \|_{L^{2}(\mathbb{R}^{4})}^{2}.
\end{align*}
\end{proof}

We now turn to the proof of Theorem $\ref{t.1}$.  The proof is based on a bootstrap procedure and continuity argument.

\begin{proof}[Proof of Theorem $\ref{t.1}$]
For $t\in I$, define 
\begin{align}
\nonumber Z_{t_0}(t)&:=\lVert |x'|^{1/4}D^2u\rVert_{L_{t,x}^4([t_0,t]\times\mathbb{R}^4)}\\
&\hspace{0.2in}+\sup_{t_0\leq t'\leq t}\lVert |\nabla|^{5/2}u(t')\rVert_{L_x^2(\mathbb{R}^4)}+\lVert |\nabla|^{3/2}u_t(t')\rVert_{L_x^2(\mathbb{R}^4)}.\label{2.1}
\end{align}

We first establish a result for general intervals $I\subset I_{\textrm{max}}$, analogous to the estimate shown in Section 2 above.  For $t_0,t_1\in \mathbb{R}$ and $I=[t_0,t_1]$, define 
\begin{align*}
\tilde{A}(t_0,t_1):=\int_{t_0}^{t_1}\int_{\mathbb{R}^4} \frac{|u(t,x)|^{4}\log(3+u(t,x)^2)}{|x'|}dxdt.
\end{align*}

We claim that there exist $\epsilon>0$ and $C>0$ so that for all $I=[t_0,t_1]\subset \mathbb{R}$, if 
\begin{align}
\tilde{A}(t_0,t_1)\leq \frac{\epsilon}{\log(3+E[u_0,u_1]+Z(t_0)^2)},\label{5.1}
\end{align}
then $Z_{t_0}(t)\leq CZ_{t_0}$ for all $t\in [t_0,t_1]$, where 
\begin{align}
Z_{t_0}:=\lVert |\nabla|^{5/2}u(t_0)\rVert_{L_{x}^2(\mathbb{R}^4)}+\lVert |\nabla|^{3/2}u_t(t_0)\rVert_{L_{x}^2(\mathbb{R}^4)}.\label{2.2}
\end{align}

To see this claim, fix $t_0\leq t_1$ so that $[t_0,t_1]\subset I_{\textrm{max}}$, and suppose that (\ref{5.1}) is satisfied.  Let $t\in [t_0,t_1]$ be given.  Then, by the Strichartz estimate of Corollary \ref{c1}, we have
\begin{align}
Z_{t_0}(t_1)\lesssim Z_{t_0}(t_0)+\| |x'|^{-1/4} \nabla_{x',y}^{2} (u^{3} \log(3 + u^{2})) \|_{L_{t,x',y}^{4/3}}.\label{1.35}
\end{align}

By the product rule, 
\begin{align}
\nonumber &\nabla_{x',y}^{2} (u^{3} \log(3 + u^{2}))\\
\nonumber &\hspace{0.2in}= 3u^{2} \log(3 + u^{2}) \nabla_{x',y}^{2} u + 6u \log(3 + u^{2}) |\nabla_{x',y} u|^{2} \\
&\hspace{0.4in}+ \frac{14 u^{3}}{(3 + u^{2})} |\nabla_{x',y} u|^{2} + \frac{2u^{4}}{(3 + u^{2})} \nabla_{x',y}^{2} u - \frac{4 u^{5}}{(3 + u^{2})^{2}} |\nabla_{x',y} u|^{2},\label{1.36}
\end{align}
and the right-hand side of ($\ref{1.35}$) is bounded by a multiple of
\begin{align*}
&Z_{t_0}+\lVert |x'|^{-1/4}u^2\log(3+u^2)\nabla_{x',y}^2u\rVert_{L_{t,x',y}^{4/3}}+\lVert |x'|^{-1/4}u\log(3+u^2)|\nabla_{x',y}u|^2\rVert_{L_{t,x',y}^{4/3}}\\
&\hspace{0.2in}+\left\lVert \frac{u^3}{|x'|^{1/4}(3+u^2)}|\nabla_{x',y}u|^2\right\rVert_{L_{t,x',y}^{4/3}}+\left\lVert \frac{u^4}{|x'|^{1/4}(3+u^2)}\nabla_{x',y}^2u\right\rVert_{L_{t,x',y}^{4/3}}\\
&\hspace{0.2in}+\left\lVert \frac{u^5}{|x'|^{1/4}(3+u^2)^2}|\nabla_{x',y}u|^2\right\rVert_{L_{t,x',y}^{4/3}}.
\end{align*}

Now, by Sobolev embedding,
\begin{align}
\nonumber &\| |x'|^{-1/4} u^{2} \log(3 + u^{2}) \nabla_{x',y}^{2} u \|_{L_{t,x',y}^{4/3}} + \left\| |x'|^{-1/4} \frac{u^{4}}{(3 + u^{2})} \nabla_{x',y}^{2} u \right\|_{L_{t,x',y}^{4/3}} \\
\nonumber &\hspace{0.2in}\lesssim \left\| \frac{u^{2} \log(3 + u^{2})^{1/2}}{|x'|^{1/2}} \right\|_{L_{t,x',y}^{2}} \| |x'|^{1/4} \nabla_{x',y}^{2} u \|_{L_{t,x',y}^{4}} \| \log(3 + u^{2})^{1/2} \|_{L_{t,x',y}^{\infty}} \\
\nonumber &\hspace{0.2in}\lesssim Z_{t_0}(t_1)\log(3 + Z_{t_0}(t_1) + E[u_0,u_1])^{1/2} \left\| \frac{u^{2} \log(3 + u^{2})^{1/2}}{|x'|^{1/2}} \right\|_{L_{t,x',y}^{2}(I \times \mathbb{R}^{4})}.
\end{align}

Next, we estimate
\begin{align}
\nonumber &\left\| |x'|^{-1/4} \frac{u^{3}}{3 + u^{2}} |\nabla_{x',y} u|^{2} \right\|_{L_{t,x',y}^{4/3}} + \left\| |x'|^{-1/4} \frac{u^{5}}{(3 + u^{2})^{2}} |\nabla_{x',y} u|^{2} \right\|_{L_{t,x',y}^{4/3}}\\
&\hspace{0.2in} \lesssim \left\| \frac{u}{|x'|^{1/4}} \right\|_{L_{t,x',y}^{4}} \| \nabla_{x',y} u \|_{L_{t,x',y}^{4}}^{2}.\label{1.38}
\end{align}
Let $1 \leq j \leq 4$ be a fixed index. Integrating by parts,
\begin{align}
\nonumber\iiint (\partial_{j} u)^{4} dx' dy dt&= -3 \iiint u (\partial_{j}^{2} u) (\partial_{j} u)^{2} dx' dy dt \\
\nonumber &\lesssim \| \nabla_{x',y} u \|_{L_{t,x',y}^{4}}^{2} \| |x'|^{1/4} \nabla_{x',y}^{2} u \|_{L_{t,x',y}^{4}} \| |x'|^{-1/4} u \|_{L_{t,x',y}^{4}}.
\end{align}
Summing over $1 \leq j \leq 4$,
\begin{align}
\nonumber \| \nabla_{x',y} u \|_{L_{t,x',y}^{4}}^{2}&\lesssim \| |x'|^{1/4} \nabla_{x',y}^{2} u \|_{L_{t,x',y}^{4}} \| |x'|^{-1/4} u \|_{L_{t,x',y}^{4}}\\
&\lesssim \| |x'|^{1/4} \nabla_{x',y}^{2} u \|_{L_{t,x',y}^{4}} \left\| \frac{u \log(3 + u^{2})^{1/4}}{|x'|^{1/4}} \right\|_{L_{t,x',y}^{4}}.\label{1.40}
\end{align}
and we therefore obtain
\begin{align}
\nonumber (\ref{1.38})&\lesssim \left\| \frac{u^{2} \log(3 + u^{2})^{1/2}}{|x'|^{1/2}} \right\|_{L_{t,x',y}^{2}} \| |x'|^{1/4} \nabla_{x',y}^{2} u \|_{L_{t,x',y}^{4}} \| \log(3 + u^{2})^{1/2} \|_{L_{t,x',y}^{\infty}} \\
\nonumber &\lesssim Z_{t_0}(t_1) \log(3 + Z_{t_0}(t_1) + E[u_0,u_1])^{1/2} \left\| \frac{u^{2} \log(3 + u^{2})^{1/2}}{|x'|^{1/2}} \right\|_{L_{t,x',y}^{2}(I \times \mathbb{R}^{4})}.
\end{align}

Finally, integrating by parts and using $\log(3 + u^{2})\gtrsim 1$,
\begin{align}
\nonumber &\iiint (\partial_{j} u)^{4} \log(3 + u^{2})^{1/2} dx' dy dt\\
\nonumber &\hspace{0.2in}= -3 \iiint (\partial_{j}^{2} u) (\partial_{j} u)^{2} u \log(3 + u^{2})^{1/2} dx' dy dt \\
\nonumber &\hspace{0.4in}- \iiint (\partial_{j} u)^{4} \frac{u^{2}}{(3 + u^{2})} \frac{1}{\log(3 + u^{2})^{1/2}} dx' dy dt \\
\nonumber &\hspace{0.2in}\lesssim \| \nabla_{x',y} u \|_{L_{t,x',y}^{4}}^{4} + \bigg(\| \log(3 + u^{2})^{1/8} \nabla_{x',y} u \|_{L_{t,x',y}^{4}}^{2} \\
\nonumber &\hspace{1.8in} \cdot \| |x'|^{1/4} \nabla_{x',y}^{2} u \|_{L_{t,x',y}^{4}} \left\| \frac{\log(3 + u^{2})^{1/4} u}{|x'|^{1/4}} \right\|_{L_{t,x',y}^{4}}\bigg),
\end{align}
so that, by $(\ref{1.40})$,
\begin{align}
\nonumber \| |\nabla_{x',y} u| \log(3 + u^{2})^{1/8} \|_{L_{t,x',y}^{4}}^{2} \lesssim \| |x'|^{1/4} \nabla_{x',y}^{2} u \|_{L_{t,x',y}^{4}} \left\| \frac{\log(3 + u^{2})^{1/4} u}{|x'|^{1/4}} \right\|_{L_{t,x',y}^{4}}.
\end{align}

Collecting the above estimates, we have shown
\begin{align*}
Z_{t_0}(t_1)&\lesssim Z_{t_0} + \bigg(\| |x'|^{-1/2} \log(3 + u^{2})^{1/2} u^{2} \|_{L_{t,x',y}^{2}(J \times \mathbb{R}^{4})}\\
&\hspace{1.4in}\cdot Z_{t_0}(t_1) \log(3 + Z_{t_0}(t_1) + E[u_0,u_1])^{1/2}\bigg),
\end{align*}
so that, for $\epsilon>0$ sufficiently small (independent of $t_0$ and $t_1$), ($\ref{5.1}$) implies, via a continuity argument,
\begin{align*}
Z_{t_0}(t_1) \lesssim Z_{t_0},
\end{align*}
which completes the proof of the claim.

Now, let $\epsilon$ and $C$ be as in the claim, and note that the Morawetz estimate of Proposition $\ref{p.morawetz}$ implies
\begin{align*}
\iint_{I\times\mathbb{R}^4} \frac{|u(t,x)|^4\log(3+u(t,x)^2)}{|x'|}dxdt\leq C'E[u_0,u_1]<\infty.
\end{align*}

Following \cite{Tao}, we partition the interval $I$ into finitely many consecutive intervals $J_k$, $k=1,\cdots,K$, with each $J_k$ of the form $J_k=[t_k,t_{k+1}]$, where $$\inf I=t_1<t_2<\cdots<t_{K+1}=\sup I.$$  Setting $t_1=\inf I$ and $Z=Z_{t_1}$, and noting that
\begin{align*}
\sum_{k=1}^K \frac{\epsilon}{\log(3+E[u_0,u_1]+C^{2k}Z^2)}\gtrsim \log(3+N/\log(3+Z^2)),
\end{align*}
it follows that one can form a partition $(J_k)$ with $K\lesssim (3+Z^2)^{C'E[u_0,u_1]}$ such that
\begin{align}
\int_{J_{k}\times\mathbb{R}^4} \frac{|u(t,x)|^4\log(3+u(t,x)^2)}{|x'|}dxdt \leq \frac{\epsilon}{\log(3+E[u_0,u_1]+C^{2k}Z^2)}\label{6.1}
\end{align}
for all $1\leq k\leq K$.

Setting $Z(J_k)=\sup_{t\in [t_k,t_{k+1}]} Z_{t_k}(t)$ for each $k$, we now inductively show the bound $Z(J_k)\leq C^kZ$.  Indeed, for $k=1$ this is a straightforward application of the claim we established above, with ($\ref{5.1}$) verified via the $k=1$ case of ($\ref{6.1}$).  Suppose now that $k\geq 1$ and that we have $Z(J_k)\leq C^kZ$.  This implies $$\frac{\epsilon}{\log(3+E[u_0,u_1]+C^{2k}Z^2)}\leq \frac{\epsilon}{\log(3+E[u_0,u_1]+Z(J_k)^2)},$$ so that, in view of ($\ref{6.1}$), an application of the claim for the interval $I=J_{k+1}$ gives $$Z(J_{k+1})\leq CZ_{t_k}\leq CZ(J_k)\leq C^{k+1}Z,$$ which is the desired inductive bound.

Assembling the estimates for $Z(J_k)$, we obtain
\begin{align*}
\|\, |x'|^{1/4}D^{2} u\, \|_{L_{t,x}^{4}(I \times \mathbb{R}^{4})} \leq C(E[u_0,u_1],\| u_{0} \|_{\dot{H}^{5/2}} + \| u_{1} \|_{\dot{H}^{3/2}}),
\end{align*}
which implies the desired global well-posedness and scattering result. 
\end{proof}

\section{Proof of Theorem \ref{t.2}}

In this section, we prove Theorem $\ref{t.2}$, which is the global well-posedness result for solutions $u$ to (NLW) with
log-supercritical nonlinearity ($\ref{3.1}$) under the symmetry condition ($\ref{symm2}$), i.e.
\begin{align*}
u(t, x_{1}, x_{2}, x_{3}, x_{4}) = u(t, x_{1}^{2} + x_{2}^{2}, x_{3}^{2} + x_{4}^{2}).
\end{align*}

The relevant Strichartz estimate in this setting is given by Lemma $\ref{l2.1}$.

\begin{proof}[Proof of Theorem \ref{t.2}]
This time the standard Morawetz estimate implies that if $J$ is an interval on which (NLW) is well-posed, then
\begin{align*}
\int_{J} \int \frac{1}{|x'| + |y|} u(t,x',y)^{4} \log(3 + u^{2}(t,x',y)) dx' dy dt \lesssim E[u_0,u_1].
\end{align*}

Now, proceeding as before, for any $I=[t_0,t_1] \subset I_{\textrm{max}}$ the Strichartz estimate of Lemma $\ref{l2.1}$ and the above Morawetz estimate lead to
\begin{align*}
Z_{t_0}(t_1) \lesssim Z_{t_0} + \| ||x'| + |y||^{-1/2} \log(3 + u^{2})^{1/2} u^{2} \|_{L_{t,x',y}^{2}(I \times \mathbb{R}^{4})} \\ \times \| (|x'| + |y|)^{1/4} D^{2} u \|_{L_{t,x',y}^{4}(I \times \mathbb{R}^{4})} \log(3 + \| u \|_{L_{t,x}^{\infty}(I \times \mathbb{R}^{4})}^{2})^{1/2}.
\end{align*}

Making the same argument as before implies the global well-posedness and scattering result.
\end{proof}

\appendix

\section{Quantitative Strichartz norm estimates for the energy-critical (NLW)}

In this appendix, we apply a variant of the method used to prove Theorem \ref{t.1} to analysis of the {\it energy-critical} NLW.
In particular, in Proposition \ref{p.appendix} below we obtain a partial-symmetry analog of the radial three-dimensional 
result in \cite{GSV}, where the Strichartz norm is controlled by a quantity which is polynomial in the energy (c.f. Theorem 
\ref{t.1} and Theorem \ref{t.2}, where the possibility of exponential growth comes from the slightly energy-supercritical 
nonlinearity).

As a preliminary tool for this analysis, we establish a weighted $L_x^5$ estimate for functions satisfying the symmetry 
condition (\ref{symm1}), i.e. $$u_i(x)=\tilde{u}_i(|x'|,x_4), \quad i=0,1,$$
with $$x=(x',x_4)\in\mathbb{R}^3\times\mathbb{R}.$$ 
For a related bound used in the study of explicit constructions of blow-up solutions for the three-dimensional nonlinear Schr\"odinger equation, see \cite{HolmerRoudenko}.

\begin{lemma}
\label{l.wb}
There exists $C>0$ such that
\begin{align*}
\lVert |x'|^{1/5}u(x',x_4)\rVert_{L_x^5(\mathbb{R}^4)}\lesssim \lVert \nabla u\rVert_{L_x^2(\mathbb{R}^4)}
\end{align*}
holds for all $u\in\mathcal{S}(\mathbb{R}^3\times\mathbb{R})$ satisfying the symmetry condition (\ref{symm1}).
\end{lemma}

\begin{proof}
Let $x=(x',x_4)\in\mathbb{R}^3\times\mathbb{R}$ be given.  By the usual radial Sobolev embedding (\ref{radSobolev}) applied to the map $x'\mapsto u(x',x_4)$, we have
\begin{align}
|x'|\, |u(x',x_4)|^2\lesssim \lVert \nabla_{x'}u(x',x_4)\rVert_{L_{x'}^2(\mathbb{R}^3)}^2.\label{4.1}
\end{align}
Combining this with the inequality
\begin{align*}
|u(x',x_4)|^3&\lesssim \int_{\mathbb{R}} |u(x',x_4)|^2|\partial_{x_4}u(x',x_4)|dx_4,\quad (x',x_4)\in \mathbb{R}^3\times\mathbb{R},
\end{align*}
we get
\begin{align*}
|x'|\, |u(x',x_4)|^5&\lesssim \bigg(\int_{\mathbb{R}^3} |\nabla_{x'}u(x',x_4)|^2dx'\bigg)\bigg(\int_{\mathbb{R}} |u(x',x_4)|^2|\partial_{x_4}u(x',x_4)|dx_4\bigg),
\end{align*}
for all $(x',x_4)\in\mathbb{R}^3\times\mathbb{R}$.

Integrating in $x'$ and $x_4$, we therefore obtain
\begin{align*}
\int_{\mathbb{R}}\int_{\mathbb{R}^3} |x'|\, |u(x',x_4)|^5dx'dx_4&\lesssim \lVert \nabla_{x'}u(x',x_4)\rVert_{L_x^2(\mathbb{R}^4)}^2\lVert |u|^2\partial_{x_4}u\rVert_{L_x^1(\mathbb{R}^4)}\\
&\lesssim \lVert \nabla u\rVert_{L_x^2(\mathbb{R}^4)}^3\lVert u\rVert_{L_x^4(\mathbb{R}^4)}^2\\
&\lesssim \lVert \nabla u\rVert_{L_x^2(\mathbb{R}^4)}^5,
\end{align*}
where we have used the Sobolev embedding to obtain the last inequality.
\end{proof}

We are now ready to state and prove the quantitative energy-critical result.

\begin{proposition}
\label{p.appendix}
Suppose that $u:I\times\mathbb{R}^4\rightarrow\mathbb{R}$ is a solution to (NLW) with nonlinearity $F:\mathbb{R}\rightarrow\mathbb{R}$ given by
\begin{align*}
F(u)=u^3.
\end{align*}

There exist constants $C_1,C_2>0$ such that if $u$ corresponds to initial data $(u,u_t)|_{t=0}=(u_0,u_1)\in \dot{H}^{1}(\mathbb{R}^4)\times L^{2}(\mathbb{R}^4)$, with $u_i$, $i=0,1$, satisfying the symmetry condition (\ref{symm1}), then
\begin{align*}
\lVert u(t,x)\rVert_{L_{t,x}^5(I\times \mathbb{R}^4)}\leq C_1E[u_0,u_1]^{C_2}.
\end{align*}
\end{proposition}

\begin{proof}
For any solution of (NLW) with energy-critical nonlinearity $F(u)=u^3$ which satisfies the symmetry condition (\ref{symm1}), the argument used to prove Proposition \ref{p.morawetz} shows the Morawetz-type bound
\begin{align*}
\iint_{I\times\mathbb{R}^4} \frac{u(t,x)^4}{|x'|}dtdx\lesssim E[u_0,u_1].
\end{align*}

Moreover, as a consequence of Lemma \ref{l.wb}, one has
\begin{align*}
\lVert u\rVert_{L_t^9L_x^{9/2}}&\lesssim \lVert |x'|^{1/9}u^{5/9}\rVert_{L_t^\infty L_x^9}\lVert \frac{u^{4/9}}{|x'|^{1/9}}\rVert_{L_{t,x}^9}\\
&=\lVert |x'|^{1/5}u\rVert_{L_t^\infty L_x^5}^{5/9}\lVert \frac{u}{|x'|^{1/4}}\rVert_{L_{t,x}^4}^{4/9}\\
&\lesssim E[u_0,u_1]^{7/18}
\end{align*}

For $t_0\leq t_1$, defining $Z_{t_0}(t_1)$ by
\begin{align*}
Z_{t_0}(t_1):=\lVert u\rVert_{L_t^2([t_0,t_1];L_x^8(\mathbb{R}^4))}+\lVert u\rVert_{L_t^5([t_0,t_1];L_x^5(\mathbb{R}^4))},
\end{align*}
the usual Strichartz estimates on $\mathbb{R}^4$ give
\begin{align*}
Z_{t_0}(t_1)&\lesssim E[u_0,u_1]^{1/2}+\lVert u^2\nabla u\rVert_{L_{t}^2([t_0,t_1];L_x^{8/7}(\mathbb{R}^4))}\\
&\lesssim E[u_0,u_1]^{1/2}+\lVert u\rVert_{L_t^{4}L_x^{16/3}}^2\lVert \nabla u\rVert_{L_t^\infty L_x^2}\\
&\lesssim E[u_0,u_1]^{1/2}+\Big[\lVert u\rVert_{L_t^{9}L_x^{9/2}}^{1/2}\lVert u\rVert_{L_t^2L_x^8}^{17/54}\lVert u\rVert_{L_{t,x}^5}^{5/27}\Big]^2E[u_0,u_1]^{1/2}\\
&\lesssim E[u_0,u_1]^{1/2}+\Big(\lVert u\rVert_{L_t^9L_x^{9/2}}E[u_0,u_1]^{1/2}\Big)Z_{t_0}(t_1)
\end{align*}
As a consequence, there exist constants $C>0$ and $\epsilon>0$ such that for all $t_0\leq t_1$ the condition
\begin{align}
\lVert u\rVert_{L_t^9([t_0,t_1];L_x^{9/2}(\mathbb{R}^4))}\leq \frac{\epsilon}{E[u_0,u_1]^{1/2}}\label{3.2}
\end{align}
implies
\begin{align*}
Z_{t_0}(t_1)\leq CE[u_0,u_1]^{1/2}.
\end{align*}

Now, invoking an iterative argument based on partitioning the interval $I$ into $K\lesssim E[u_0,u_1]^8$ intervals $[t_k,t_{k+1}]$, $k=1,\cdots,K$, on which (\ref{3.2}) is satisfied, we obtain
\begin{align*}
\lVert u\rVert_{L_{t}^5([t_k,t_{k+1}];L_x^5(\mathbb{R}^4))}\leq Z_{t_k}(t_{k+1})\lesssim E[u_0,u_1]^{1/2}
\end{align*}
for each $1\leq k\leq K$, and thus
\begin{align*}
\lVert u\rVert_{L_t^5(I;L_x^5(\mathbb{R}^4))}&\lesssim E[u_0,u_1]^{21/10}
\end{align*}
as desired.
\end{proof}

\end{document}